\documentclass[12pt]{article}
\usepackage[centertags]{amsmath}
\usepackage{amsfonts}
\usepackage{amssymb}
\usepackage{amsthm}

\usepackage{graphicx}
\usepackage{enumerate}
\usepackage{mathrsfs}
\usepackage{latexsym}
\usepackage{psfrag}
\usepackage{mathrsfs}

\usepackage{pgf,tikz}%tikz
\usepackage{multicol}

\theoremstyle{plain}
\newtheorem{theorem}{Theorem}[section]
\newtheorem{lemma}[theorem]{Lemma}
\newtheorem{proposition}[theorem]{Proposition}
\newtheorem{corollary}[theorem]{Corollary}

\theoremstyle{definition}
\newtheorem{definition}[theorem]{Definition}

\newtheorem{example}[theorem]{Example}

\newcommand{\comments}[1]{}

\begin{document}

\title{On the degree pairs of a graph}

\author{
Yu-pei Huang
\thanks{College of Applied Mathematics, Beijing Normal University, Zhuhai 519087, China.
{\tt Email: yphuang@bnuz.edu.cn}}
\and
Chia-an Liu
\thanks{Department of Mathematical Sciences, University of Delaware Newark, Delaware 19716, U.S.A.
{\tt Email: liuchiaan8@gmail.com}}
\and
Chih-wen Weng
\thanks{Department of Applied Mathematics, National Chiao Tung University, Hsinchu 30010, Taiwan.
{\tt Email: weng@math.nctu.edu.tw}}
}

\date{\small November 6, 2018}

\maketitle

\begin{abstract}
Let $G$ be a simple graph without isolated vertices. For a vertex $i$ in $G$, the degree $d_i$ is the number of vertices adjacent to $i$ and the
average $2$-degree $m_i$ is the mean of the degrees of the vertices which are adjacent to $i$. The sequence of pairs $(d_i, m_i)$ is called the sequence of degree pairs of $G$.
We provide some necessary conditions for a sequence of real pairs $(a_i, b_i)$ of length $n$ to be the degree pairs of a graph of order $n$. A graph $G$ is called pseudo $k$-regular if $m_i=k$ for every vertex $i$ while $d_i$ is not a constant. Let $N(k)$ denote the minimum number of vertices in a pseudo $k$-regular graph.
We utilize the above necessary conditions to find all pseudo $3$-regular graphs of orders no more than $10,$ and all pseudo $k$-regular graphs of order $N(k)$ for $k$ up to $7$. We give bounds of $N(k)$ and show that $N(k)$ is at most $k+6.$
%We gives bounds of $N(k)$ and show that
%the ratio $N(2k+1)/(2k+1)$ tends to $1$ as $k$ tends infinity.
\end{abstract}

{\noindent} Keywords: Degree, average $2$-degree, harmonic graphs, pseudo regular graphs.

{\noindent} 2010 MSC: 05C07.

\section{Introduction}
Throughout this paper, let $G$ be a simple graph with vertex set $V(G)=\{1, 2, \ldots, n\}$ and edge set $E(G)$ which has no isolated vertices. Let $d_i$ and $m_i$ be the degree and \emph{average $2$-degree} of  vertex $i\in V(G)$, where $d_i=\sum_{j:ji\in E(G)}1$ and $m_i={d_i}^{-1}\sum_{j:ji\in E(G)}d_j.$  The sequence $(d_i, m_i)_{i=1}^n$ of pairs is called the sequence of \emph{degree pairs} of $G$.
\medskip

Degree sequences of graphs have been well-studied for a long time, e.g., \cite{60:eg}.
The additional average $2$-degree sequence $m_i$ of graphs will provide more information about a graph.
Indeed the sequence of degree pairs $(d_i,m_i)_{i=1}^n$ has already appeared in literature. For example in the extremal spectral theory, many upper bounds of eigenvalues are expressed by a function of the sequence \cite{11:cpz,04:d,14:hw,01:lp,98:m,04:z}. From application aspects, the parameters $d_i$ and $m_i$ are natural in the study of a two-way communication network \cite{blj:17}, especially in the investigation of the distance-$2$ structure of the network, in which the simple graph model is employed with vertices representing the nodes and edges the links. The number $d_i$ of links at node $i$ is easy acquainted by the node $i$ itself, while the number $d_j$ of a node $j$ that is adjacent to $i$ might not be obtained easily, but at least node $i$ has a rough idea of the mean number $m_i$ of $d_j$, especially when this number happened to be an integer from some other network structure information. Other related parameters which are introduced recently in the study of distance-$2$ structure of a graph can be found in \cite{chh:17}.
\medskip

The sequence of degree pairs might have applications in cryptography as the factorization of an integer does.
It is easy for a graph (resp. a pair of prime numbers) to generate its sequence of degree pairs (resp. their product), but much harder for the reverse. Hence the sequence of degree pairs of a graph can track the graph while easily keeping secret of the exact information of the graph. Unfortunately two different graphs might have the same sequence of degree pairs.
For instance, the two graphs in  Figure 1 are not isomorphic, but have the same sequence $(2, 3)$, $(3, 3)$, $(3, 3)$, $(4, 3)$, $(3, 3)$, $(3, 3)$, $(2, 3)$  of degree pairs.
\begin{center}

%\begin{multicols}{2}
\begin{picture}(100,60)
\put(10,30){\circle*{3}}
\put(30,10){\circle*{3}}    \put(30,50){\circle*{3}}
\put(50,30){\circle*{3}}
\put(70,10){\circle*{3}}    \put(70,50){\circle*{3}}
\put(90,30){\circle*{3}}
\qbezier(50,30)(40,20)(30,10) \qbezier(50,30)(40,40)(30,50)
\qbezier(50,30)(60,20)(70,10) \qbezier(50,30)(60,40)(70,50)
\qbezier(10,30)(20,20)(30,10) \qbezier(10,30)(20,40)(30,50)
\qbezier(90,30)(80,20)(70,10) \qbezier(90,30)(80,40)(70,50)
\qbezier(30,10)(50,10)(70,10) \qbezier(30,50)(50,50)(70,50)
\end{picture}
~~~~~~~~~~
\begin{picture}(100,60)
\put(10,30){\circle*{3}}
\put(30,10){\circle*{3}}    \put(30,50){\circle*{3}}
\put(50,30){\circle*{3}}
\put(70,10){\circle*{3}}    \put(70,50){\circle*{3}}
\put(90,30){\circle*{3}}
\qbezier(50,30)(40,20)(30,10) \qbezier(50,30)(40,40)(30,50)
\qbezier(50,30)(60,20)(70,10) \qbezier(50,30)(60,40)(70,50)
\qbezier(10,30)(20,20)(30,10) \qbezier(10,30)(20,40)(30,50)
\qbezier(90,30)(80,20)(70,10) \qbezier(90,30)(80,40)(70,50)
\qbezier(30,10)(30,30)(30,50) \qbezier(70,10)(70,30)(70,50)
\end{picture}

%\end{multicols}

\bigskip
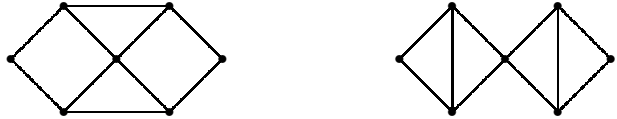
{\bf Figure 1:} Two graphs with the same sequence of degree pairs.
\end{center}
On the other hand, there are also graphs that are uniquely determined by their sequences of  degree pairs. Two examples are provided in Figure 2.

\begin{center}
\begin{picture}(100, 70)
\put(10,10){\circle*{3}}
\put(-5,-5){$(1,3)$}
\put(50,10){\circle*{3}}    \put(50,50){\circle*{3}}
\put(35,-5){$(3,\frac{5}{3})$}  \put(75,-5){$(2,\frac{5}{2})$}
\put(90,10){\circle*{3}}    \put(90,50){\circle*{3}}
\put(95,50){$(2,2)$}  \put(20,50){$(2,\frac{5}{2})$}
\qbezier(10,10)(30,10)(50,10)   \qbezier(50,10)(50,30)(50,50)
\qbezier(50,50)(70,50)(90,50)   \qbezier(90,50)(90,30)(90,10)
\qbezier(50,10)(70,10)(90,10)
\end{picture}
~~~~~~~~~~~~
\begin{picture}(100,60)
\put(10,10){\circle*{3}}
\put(-5,-5){$(1,2)$}
\put(50,10){\circle*{3}}    \put(50,50){\circle*{3}}
\put(35,-5){$(2,2)$}  \put(75,-5){$(2,\frac{5}{2})$}
\put(90,10){\circle*{3}}    \put(90,50){\circle*{3}}
\put(95,50){$(2,\frac{5}{2})$}  \put(20,50){$(3,2)$}
\qbezier(10,10)(30,10)(50,10)   \qbezier(50,10)(50,30)(50,50)
\qbezier(50,50)(70,50)(90,50)   \qbezier(90,50)(90,30)(90,10)
\qbezier(50,50)(70,30)(90,10)
\end{picture}
%\end{multicols}

\bigskip
\bigskip

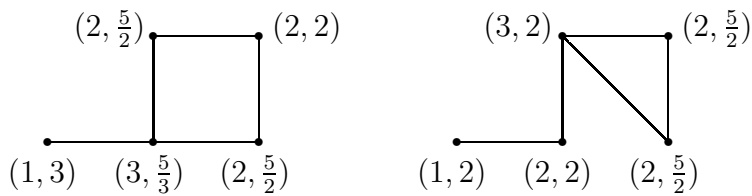
{\bf Figure 2:} Two graphs uniquely determined by their sequences of degree pairs.
\end{center}

\bigskip

A graph $G$ is called $k$-\emph{harmonic} if $m_i=k$ for all vertices $i$.   Dress and Gutman showed that for a graph $G$ and a  positive integer $k$, the number of walks of length $\ell\geq 1$ in $G$ equals $2|E(G)|k^{\ell-1}$ if and only if $G$ is $k$-harmonic \cite{03:dgr}.
They also proved that $k$ is an integer in a $k$-harmonic graph. More families of graphs involved with the harmonic concept are proposed in~\cite{02:g,03:bggp,03:dgr}. The class of $k$-regular graphs, i.e. $d_i=k$ for all vertices $i$, is a subclass of $k$-harmonic graphs. A graph is \emph{pseudo $k$-regular} if it is $k$-harmonic but not $k$-regular.
Let $N(k)$ denote the minimum number of vertices in a pseudo $k$-regular graph.
\medskip

The paper is organized as follows. Some necessary conditions for a sequence of real pairs to be degree pairs are
given in Section~\ref{sec_deg_pairs}. In Section~\ref{indep}, we provide a lower bound of the independent number of the square graph $G^2$ of $G$ in terms of the sequence of degree pairs of $G$.
In Section~\ref{sec_pse_kreg}, we determine all pseudo $3$-regular graphs of order no more than $10$.
We aim at the study of minimal pseudo regular graphs. In Section~\ref{sec_min_prg}, we determine all pseudo $k$-regular graphs of order $N(k)$ for $k$ up to $7$, give bounds of $N(k)$ and construct pseudo $k$-regular graphs of order $k+4$ and $k+6$ for odd and even $k,$ respectively.
%show that the ratio $N(2k+1)/(2k+1)$ tends to $1$ as $k$ tends infinity.
Open problems are provided in Section~\ref{sec_open} as a concluding remark of this article.

\bigskip

\section{The necessary conditions}       \label{sec_deg_pairs}
Some necessary conditions of a sequence of real pairs to be degree pairs are
 given in this section.

\subsection{Criteria on the sequence of degree pairs}

Let $G$ be a graph with degree pairs $\{(d_i, m_i)\}_{i=1}^n.$  We begin with a feasible condition of degree pairs.

\begin{lemma}\label{lem_sum_dimi}
$$\sum_{i=1}^n d_im_i = \sum_{i=1}^n d_i^2.$$
\end{lemma}
\begin{proof}
It follows from
\begin{eqnarray}
\sum_{i=1}^nd_im_i&=&\sum_{i=1}^nd_i \left(d_i^{-1} \sum_{j:ji\in E(G)}d_j\right)
\nonumber   \\
&=&\sum_{j=1}^n\sum_{i:ij\in E(G)}d_j=\sum_{j=1}^nd_j^2.
\nonumber
\end{eqnarray}
\end{proof}

\medskip

Like a well-known property of degree sequence stating that the number of odd values $d_i$ is even, we have a similar version for the number of odd values $d_im_i$.

\begin{corollary}\label{cor_even}
There are even number of odd values $d_im_i$ for $1\leq i\leq n.$
\end{corollary}
\begin{proof}
Since $\sum_{i=1}^nd_i=2|E(G)|$ is even, there are even number of odd $d_i,$ and so does $d_i^2.$ By Lemma~\ref{lem_sum_dimi}, $\sum_{i=1}^n d_im_i$ is even and we have the proof.
\end{proof}

\medskip

\begin{corollary}
$$\sum_{i=1}^nm_i^2 \geq \sum_{i=1}^nd_i^2$$
with equality if and only if $m_i=d_i$ for all $i.$
\end{corollary}
\begin{proof}
It immediately follows by Cauchy's inequality that
\begin{equation}    \label{eq_mi2_geq_di2}
\big(\sum_{i=1}^nd_i^2\big)\big(\sum_{i=1}^n m_i^2\big) \geq \big(\sum_{i=1}^n d_im_i\big)^2 = \big(\sum_{i=1}^n d_i^2\big)^2
\end{equation}
with equality if and only if $m_i=cd_i$ for all $1\leq i\leq n$ and some constant $c,$ where the last equality in~\eqref{eq_mi2_geq_di2} is from Lemma~\ref{lem_sum_dimi}. Again, by Lemma~\ref{lem_sum_dimi}, $c$ must be 1.
\end{proof}

\medskip

Degrees give hints of graph properties, e.g., $\sum_{i=1}^nd_i=2|E(G)|$ gives the number of edges.  Degree pairs provide more information about the graph structure.
\begin{proposition}
 If
$$\max_{1\leq i\leq n}d_im_i\geq n$$
then the girth of $G$ is at most $4.$
\end{proposition}
\begin{proof}
For each $1\leq i\leq n,$ let $G_1(i)$ and $G_2(i)$ denote the set of vertices in $V(G)$ that have distance $1$ and $2$ to $i,$ respectively. If the graph has girth at least $5$ then
$$|V(G)|-1\geq |G_1(i)|+|G_2(i)|=d_im_i\qquad (1\leq i\leq n),$$
a contradiction.
\end{proof}

Reviewing the graphs given in Figure 2,  since both of them are of $\max_{1\leq i\leq 5}d_im_i\geq |V(G)|=5,$ there must exist some $3$-cycle $C_3$ or $4$-cycle $C_4$ in the graph by the above result. The uniqueness of the two graphs can be quickly confirmed.
\medskip

\subsection{An analogue of the Erd\H{o}s-Gallai Theorem}
The Erd\H{o}s-Gallai Theorem~\cite{60:eg} is a well-known solution to the graph realization problems which gives a necessary and sufficient condition for a sequence of nonnegative integers to be the degree sequence (i.e. \emph{graphic}) on a finite simple graph.
\begin{lemma}[Erd\H{o}s-Gallai Theorem]         \label{lem_EGThm}
A sequence of nonnegative integers $d_1\geq d_2\geq \cdots\geq d_n$ can be represented as the degree sequence of a finite simple graph on $n$ vertices if and only if $\sum_{i=1}^n d_i$ is even and
$$\sum_{i=1}^k d_i \leq k(k-1) + \sum_{i=k+1}^n \min\{d_i,k\}\qquad (1\leq k\leq n).$$
\qed
\end{lemma}

\medskip

Lemma~\ref{lem_EGThm} was shown by several authors. To obtain a quick proof, one can see~\cite{86:c}. We give an analogue of Erd\H{o}s-Gallai Theorem for the degree pairs in the following which is, unfortunately, only a necessary condition. However, it is still helpful to determine a non-graphic degree pair sequence before we try to graph it. For real numbers $a_1,b_1,a_2,b_2,$ we define ordered pairs $(a_1,b_1)\succeq (a_2,b_2)$ if $a_1>a_2,$ or $a_1=a_2$ and $b_1\geq b_2.$ It is clear that $(\mathbb{R}^2,\succeq)$ is totally ordered.
\begin{proposition}     \label{prop_AnalogueEG}
If a sequence of ordered pairs of positive real numbers $(d_1,m_1)\succeq (d_2,m_2)\succeq \cdots\succeq (d_n,m_n)$ is a sequence of degree pairs  of a simple graph $G$ of order $n$, then
\begin{itemize}
\item[(i)]      $d_i$ and $d_im_i$ are both positive integers, for $i=1,2,\ldots,n,$
\item[(ii)]     $d_im_i\leq \sum_{j=1}^{d_i+1}d_j-d_{\min\{d_i+1,i\}},$ for $i=1,2,\ldots,n,$
\item[(iii)]    $d_im_i\geq \sum_{j=n-d_i}^{n}d_j-d_{\max\{n-d_i,i\}},$ for $i=1,2,\ldots,n,$
\item[(iv)]    $\sum_{i=1}^n d_im_i = \sum_{i=1}^n d_i^2,$
\item[(v)]     $\sum_{i=1}^n d_i$ is even (and so does $\sum_{i=1}^n d_im_i$),
\item[(vi)]     $\sum_{i=1}^k d_i \leq k(k-1) + \sum_{i=k+1}^n \min\{d_i,k\},$ for $k=1,2,\ldots,n,$ and
\item[(vii)]      $\sum_{i=1}^k d_im_i \leq \sum_{i=1}^k d_i\min\{d_i,k-1\} + \sum_{i=k+1}^n d_i\min\{d_i,k\},$ for $k=1,2,\ldots,n.$
\end{itemize}
\end{proposition}
\begin{proof}
By definition (i) immediately follows. We have (ii) and (iii) by considering the possibly maximal and minimal sums of the degrees of vertices adjacent to each vertex, respectively. Also, (iv), (v), and (vi) are directly from Lemma~\ref{lem_sum_dimi}, Corollary~\ref{cor_even}, and Lemma~\ref{lem_EGThm}, respectively. Let $S=\{i~|~1\leq i\leq k\}.$ Then the left-hand side of (vii) can be written as
$$\sum_{i=1}^k d_im_i=\sum_{i=1}^n d_i|S\cap G_1(i)|,$$
which is not more than the summation forms in right-hand side.
\end{proof}

One should notice that the conditions listed in Proposition~\ref{prop_AnalogueEG} is not sufficient. For example, the sequence of pairs $(d_i,m_i)_{i=1}^7=(4,3),(3,3),(3,3),(3,3),(3,3),(1,3),(1,3)$ satisfies all of them. Nevertheless, one can work by hand to see that it is not a sequence of degree pairs of any graph of order $7$.
\medskip

\section{A lower bound for the independence number of $G^2$}\label{indep}

The \emph{independent number} $\alpha(G)$ of a graph $G$ is the maximum size of a vertex subset consisting of pairwise
nonadjacent vertices.  Let $G^2$ be the \emph{square} of $G,$ i.e.
$$V(G^2) = V(G)~~~\text{and}~~~E(G^2) = \{ij~\mid~d(i,j)=1~\text{or}~2\},$$
where $d(i,j)$ denotes the distance of vertices $i$ and $j$ in the graph $G.$ The independent number $\alpha(G^2)$ of $G^2$ applies to solve data aggregation problem and collision avoidance problem in a wireless sensor network $G$ \cite{km:08}. Using the probabilistic method~\cite[Theorem~18.4]{11:j}, we have the following result.
\begin{proposition}
Let $G$ be a simple graph with no isolated vertices and of degree pair sequence $(d_i,m_i)_{i=1}^n.$ Then the independence number of the square $G^2$ of $G$ satisfies
$$\alpha(G^2)\geq \sum_{i=1}^n\frac{1}{1+d_im_i}.$$
\end{proposition}
\begin{proof}
A subset $S$ of $\{1, 2, \ldots, n\}$ such that  $i\in S$ implies $G_1(i) \subseteq \{i+1, i+2, \ldots, n\}$ is an independent set. We use the probabilistic method to expected size of $S$. If every vertex is picked equally in random then the probability of a vertex $i$ appears before those vertices in $G_1(i)\cup G_2(i)$ is $p_i=(1+|G_1(i)|+|G_2(i)|)^{-1}$, which is at least $(1+d_im_i)^{-1}$. Moreover, the expected size of a set consisting of these $i$ is $\sum_{i=1}^n p_i$ because of the linearity of expectation. Therefore,
$$\alpha(G^2)\geq \sum_{i=1}^n p_i\geq \sum_{i=1}^n\frac{1}{1+d_im_i}.$$
\end{proof}

\bigskip

\section{Pseudo regular graphs}     \label{sec_pse_kreg}
This section turns to the study of pseudo regular graphs. We start from the properties of harmonic graphs. Recall that a simple graph $G$ with no isolated vertices is $k$-harmonic if its average $2$-degree $m_i=k$ for every $i\in V(G),$ which was introduced in~\cite{03:dgu}. From the definition of a $k$-harmonic graph, $k$ is a rational number, but indeed $k$ is an integer.
\begin{proposition}(\cite{03:dgu})
If $G$ is $k$-harmonic then $k$ is a positive integer.
\end{proposition}
\begin{proof}
Let $A$ be the adjacency matrix of $G.$ Then we have $Ad=kd$,
where $d$ is the column vector $(d_1,d_2,\ldots,d_n)^T$.
Being a zero of the characteristic polynomial of $A,$ $k$ is an algebraic integer. Since $k$ is also a positive rational number, $k$ is a positive integer.
\end{proof}

\medskip

It is natural to ask when a $k$-harmonic graph with the given order $n$ attains the maximum number of edges.
\begin{proposition}
A $k$-harmonic graph on $n$ vertices has at most $nk/2$ edges, and the maximum is obtained if and only if the graph is regular.
\end{proposition}
\begin{proof}  Let $G$ be a $k$-harmonic graph with degree pairs $\{(d_i, m_i)\}_{i=1}^n$, where $m_i=k$.
By Cauchy's inequality,
$$2k|E(G)|=\sum_{i=1}^nd_im_i = \sum_{i=1}^nd_i^2 \geq \frac{(\sum_{i=1}^nd_i)^2}{n} = \frac{4|E(G)|^2}{n},$$
we have $|E(G)|\leq nk/2$ and the equality is obtained if and only if $d_i$ is a constant.
\end{proof}

\medskip

The next is to ask when a connected $k$-harmonic graph of order $n$ attains the minimal number of edges. This was done in
\cite{03:dgu} with the following trees.
\begin{definition}      \label{defn_Tk}
For each $k\geq 2,$ let $T_k$ be the tree of order $k^3-k^2+k+1$ whose root has degree $k^2-k+1$ and each neighbor of the root has $k-1$ children as leafs.
\end{definition}
\begin{center}
%\begin{multicols}{2}
\begin{picture}(120,100)
\put(10,10){\circle*{3}}    \put(10,50){\circle*{3}}
\put(60,10){\circle*{3}}    \put(60,50){\circle*{3}}    \put(60,90){\circle*{3}}
\put(110,10){\circle*{3}}   \put(110,50){\circle*{3}}
\qbezier(10,10)(10,30)(10,50)       \qbezier(10,50)(35,70)(60,90)
\qbezier(60,10)(60,30)(60,50)       \qbezier(60,50)(60,70)(60,90)
\qbezier(110,10)(110,30)(110,50)    \qbezier(110,50)(85,70)(60,90)
\end{picture}
~~~~~~~~~~~~
\begin{picture}(200,100)
\put(0,10){\circle*{3}}     \put(20,10){\circle*{3}}
\put(30,10){\circle*{3}}    \put(50,10){\circle*{3}}
\put(60,10){\circle*{3}}    \put(80,10){\circle*{3}}
\put(90,10){\circle*{3}}    \put(110,10){\circle*{3}}
\put(120,10){\circle*{3}}   \put(140,10){\circle*{3}}
\put(150,10){\circle*{3}}   \put(170,10){\circle*{3}}
\put(180,10){\circle*{3}}   \put(200,10){\circle*{3}}
\put(10,50){\circle*{3}}    \put(40,50){\circle*{3}}    \put(70,50){\circle*{3}}    \put(100,50){\circle*{3}}
\put(130,50){\circle*{3}}   \put(160,50){\circle*{3}}   \put(190,50){\circle*{3}}   \put(100,90){\circle*{3}}
\qbezier(0,10)(5,30)(10,50)         \qbezier(20,10)(15,30)(10,50)
\qbezier(30,10)(35,30)(40,50)       \qbezier(50,10)(45,30)(40,50)
\qbezier(60,10)(65,30)(70,50)       \qbezier(80,10)(75,30)(70,50)
\qbezier(90,10)(95,30)(100,50)      \qbezier(110,10)(105,30)(100,50)
\qbezier(120,10)(125,30)(130,50)    \qbezier(140,10)(135,30)(130,50)
\qbezier(150,10)(155,30)(160,50)    \qbezier(170,10)(165,30)(160,50)
\qbezier(180,10)(185,30)(190,50)    \qbezier(200,10)(195,30)(190,50)
\qbezier(10,50)(55,70)(100,90)     \qbezier(40,50)(70,70)(100,90)
\qbezier(70,50)(85,70)(100,90)     \qbezier(100,50)(100,70)(100,90)
\qbezier(130,50)(115,70)(100,90)   \qbezier(160,50)(130,70)(100,90)
\qbezier(190,50)(145,70)(100,90)
\end{picture}
%\end{multicols}

\bigskip
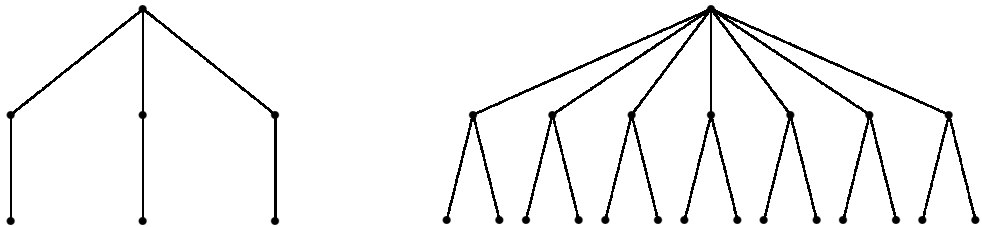
{\bf Figure 3:} The trees $T_2$ and $T_3$.
\end{center}
\medskip

\begin{proposition}\label{dgu}(\cite{03:dgu}) The $k$-harmonic trees are the trees $T_k$.
\end{proposition}
\medskip

We now turn to the study of  connected pseudo $k$-regular graphs.
\begin{lemma}\label{t2}
If $G$ is connected pseudo $2$-regular then $G$ is the tree $T_2.$
\end{lemma}
\begin{proof}
If there exists a vertex $i\in V(G)$ of degree $d_i\geq 4,$ then each vertex $j$ adjacent to $i$ has degree $j>2,$ or otherwise the average $2$-degree $m_j$ of $j$ is more than $2.$ However, a contradiction occurs since the average $2$-degree $m_i$ of $i$ is more than $2.$ Hence the maximum degree of $G$ is $3.$ Suppose $i\in V(G)$ of degree $d_i=3,$ then each vertex $j$ adjacent to $i$ is of degree $j\neq 1$ (or otherwise $j$ has average $2$-degree $j>2$).
Also, $d_j\neq 3$ since $m_i=2.$ Hence $d_j=2.$ It implies that $G$ is exactly $T_2.$
\end{proof}

\medskip

We want to find the maximum degree of a connected pseudo $k$-regular graph of order $n$.
It turns out that $T_k$ gives the maximum degree. We need a lemma in more general setting.

\begin{lemma}   \label{lem_djmimj}
 Let $G$ be a graph with degree pairs $\{(d_i, m_i)\}_{i=1}^n$.
Then for any $ij\in E(G)$ with $d_j\leq m_i,$ we have
$$d_i\leq m_i(m_j-1)+1$$
with equality if and only if $d_j=m_i$ and all neighbors of $j$ except for $i$ have degree $1.$
\end{lemma}
\begin{proof}
By definition,
$$d_jm_j=\sum_{k:kj\in E(G)}d_k \geq d_i+1\cdot(d_j-1)$$
Therefore,
$$d_i\leq d_j(m_j-1)+1 \leq m_i(m_j-1)+1$$
with the first equality if and only if all neighbors of $j$ except for $i$ have degree $1,$ and $d_j=m_i$ for the latter equality.
\end{proof}

\medskip

The following theorem includes Proposition~\ref{dgu} as a special case.

\begin{theorem}     \label{thm_Tk}
Let $k\geq 2$ be a positive integer and $G$ be a connected simple graph with average $2$-degree $m_i\leq k$ for each vertex $i\in V(G).$ Then the maximum degree
$$\Delta(G)\leq k^2-k+1.$$
Moreover, the following (i)-(iv) are equivalent.
\begin{itemize}
\item[(i)]      $\Delta(G)=k^2-k+1.$
\item[(ii)]     $G$ is the tree $T_k.$
\item[(iii)]    $G$ is a pseudo $k$-regular tree.
\item[(iv)]     $G$ has a vertex $j$ of degree and average $2$-degree $d_j=m_j=k$ and all neighbors of $j$ have degree $1$ with exactly one exception.
\end{itemize}
\end{theorem}
\begin{proof}
For every vertex $i\in V(G),$ it is obvious that there exists $ji\in E(G)$ such that $d_j\leq m_i.$ By Lemma~\ref{lem_djmimj},
$$d_i\leq m_i(m_j-1)+1\leq k^2-k+1$$
and so does $\Delta(G),$ with equality if and only if each vertex $j$ adjacent to $i$ of degree $d_i=k^2-k+1$ must have degree $d_j=k$ and all neighbors of $j$ except for $i$ have degree $1.$ Hence (i) and (ii) are equivalent. Additionally, the implication (ii)$\Rightarrow$(iii) is clear. We show that (iii) implies (iv) and (iv) implies (i).

\smallskip

To prove that (iii) implies (iv), let $G$ be a pseudo $k$-regular tree containing a longest path $P = v_1,v_2,v_3,\ldots,v_\ell.$ Note that $\ell\geq 4$ since $k\geq 2.$ Then the end $v_1$ is a leaf, and hence $v_2$ is a vertex of degree and average $2$-degree $d_{v_2}=m_{v_2}=k.$ Since $P$ is of maximum length in $G,$ all neighbors of $v_2$ except for $v_3$ are leafs, and $v_2$ is a vertex satisfying the conditions. To prove that (iv) implies (i), assume (iv) holds. Then the exceptional neighbor of $j$ has degree $k^2-k+1$ since $d_jm_j=k^2.$ The proof is complete.
\end{proof}

\medskip

If $G$ is a pseudo $k$-regular graph, then the unique neighbor of a vertex of degree $1$ has degree $k$ in $G.$ We have seen in the previous proof that any neighbor of a vertex of degree $k^2-k+1$ also has degree $k.$ One might ask what other vertices have their neighbors of the same degree $k.$
\begin{lemma}       \label{lem_k^2-3k+4}
Let $G$ be a pseudo $k$-regular graph where $k\geq 3,$ and $ij$ be an edge in $G$ with $2\leq d_j<k.$ Then
$$2\leq d_i\leq k^2-3k+4,$$
with the second equality if and only if $d_j=k-1$ and all neighbors of $j$ except for $i$ have degree $d_j=2.$
\end{lemma}
\begin{proof}
Since $d_j<k,$ $d_i\neq1$ which implies the first inequality. Moreover, all neighbors of $j$ have degree more than $1.$ Hence
$$d_i+2(d_j-1)\leq d_jm_j = d_jk,$$
and we have
$$d_i\leq (k-2)d_j+2 \leq k^2-3k+4.$$
The result follows by observing the above inequalities.
\end{proof}

\medskip

\begin{corollary}       \label{cor_k^2-3k+5}
Let $G$ be a pseudo $k$-regular graph of order $n$ with a vertex $i$ of degree $d_i\geq k^2-3k+5.$ Then
\begin{itemize}
\item[(i)]      every neighbor $j$ of $i$ has degree $d_j=k,$ and
\item[(ii)]     the order of $G$ is at least
        $$f(k):=\left\lceil~\frac{5k^4-31k^3+94k^2-140k+100}{k^2}~\right\rceil.$$
\end{itemize}
\end{corollary}
\begin{proof}
To prove (i), since $k^2-3k+5$ clearly greater than $k,$ any vertex $j$ adjacent to $i$ has degree more than $1.$ By the contrapositive statement of Lemma~\ref{lem_k^2-3k+4}, we have $d_j\geq k.$ The argument is true for all neighbors of $i,$ and hence $d_j=k.$
To prove (ii), let $S=V(G)\setminus(\{i\}\cup G_1(i)).$  By Lemma~\ref{lem_sum_dimi} and (i),
$$d_ik+d_ik^2+\sum_{j\in S} d_jk = d_i^2+d_ik^2+\sum_{j\in S} d_j^2.$$
Consequently,
\begin{eqnarray}
(k^2-3k+5)(k^2-4k+5)&\leq& d_i(d_i-k)=\sum_{j\in S}d_j(k-d_j)
\nonumber   \\
&\leq& \left(\frac{k}{2}\right)^2(n-1-(k^2-3k+5)),
\label{eq_quad_ineq}
\end{eqnarray}
where~\eqref{eq_quad_ineq} is from the quadratic inequality $x(k-x)\leq (k/2)^2$ for all $x\in\mathbb{R}.$ Simplify the above inequality, and we have the proof.
\end{proof}

\medskip

The above result tells that a pseudo $3$-regular graph of maximum degree more than $4$ has number of vertices at least $f(3)=11.$
This gives hope to determine all connected pseudo $3$-regular graphs of orders up to $10$.
Recall that in Theorem~\ref{thm_Tk} we prove that $4\leq \Delta(G)\leq 7,$ and solve the case $\Delta(G)=7.$
\begin{corollary}       \label{cor_di4seq}
Let $G$ be a pseudo $3$-regular graph.
\begin{itemize}
\item[(i)]      If $i\in V(G)$ has degree $d_i=6,$ then each neighbor $j$ of $i$ has degree $d_j=3,$ and the neighbors of $j$ have degree sequence $(6,2,1).$
\item[(ii)]     If $i\in V(G)$ has degree $d_i=5,$ then each neighbor $j$ of $i$ has degree $d_j=3,$ and the neighbors of $j$ have degree sequence $(5,2,2)$ or $(5,3,1).$
\item[(iii)]    If $i\in V(G)$ has degree $d_i=4,$ then the neighbors of $i$ have degree sequence $(3,3,3,3),$ $(4,3,3,2),$ or $(4,4,2,2).$
\end{itemize}
\end{corollary}
\begin{proof}
They are all straightforward from Corollary~\ref{cor_k^2-3k+5}(i).
\end{proof}

\medskip

\begin{lemma}       \label{lem_a1a2a3a4}
Let $G$ be a connected pseudo $3$-regular graph with maximum degree $4$ and $a_j=|\{i\in V(G)\mid d_i=j\}|$ for $j=1,2,3,4.$ Then
\begin{itemize}
\item[(i)]      $a_1+a_2=2a_4,$
\item[(ii)]     $|V(G)| = a_3+3a_4,$
\item[(iii)]    $a_1\leq a_3,$ and
\item[(iv)]     $a_1,a_2,a_3$ have the same parity.
\end{itemize}
\end{lemma}
\begin{proof}
(i) and (ii) follow from solving the equation in Lemma~\ref{lem_sum_dimi}. (iii) follows, since there exists an injection from the set of degree $1$ vertices into the set of degree $3$ vertices defined by the adjacency. There are even number of vertices of odd degrees, so $a_1+a_3$ is even. The remaining of (iv) follows from (i).
\end{proof}

\medskip

By Proposition~\ref{prop_AnalogueEG}, Corollary~\ref{cor_di4seq}(iii), Lemma~\ref{lem_a1a2a3a4}, and checking the graphicness exhaustively, if $G$ is a pseudo $3$-regular graph of order $n\leq 10,$ the parameters $(n,a_4,a_3,a_2,a_1)$ can only be $(7,1,4,2,0),$ $(8,2,2,2,2),$ $(9,1,6,2,0),$ $(9,2,3,3,1),$ $(9,3,0,6,0),$ $(10,2,4,0,4),$ or $(10,2,4,4,0).$ Their corresponding graphs are presented in Figure 4.

\medskip

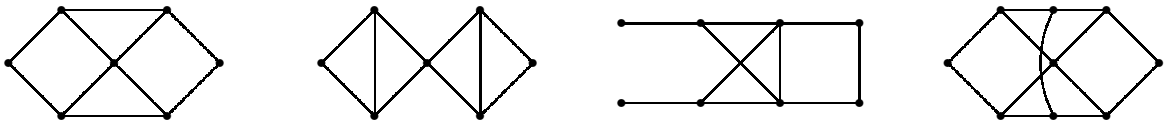
\begin{center}
%\begin{multicols}{2}
\begin{picture}(100,60)
\put(10,30){\circle*{3}}
\put(30,10){\circle*{3}}    \put(30,50){\circle*{3}}
\put(50,30){\circle*{3}}
\put(70,10){\circle*{3}}    \put(70,50){\circle*{3}}
\put(90,30){\circle*{3}}
\qbezier(50,30)(40,20)(30,10) \qbezier(50,30)(40,40)(30,50)
\qbezier(50,30)(60,20)(70,10) \qbezier(50,30)(60,40)(70,50)
\qbezier(10,30)(20,20)(30,10) \qbezier(10,30)(20,40)(30,50)
\qbezier(90,30)(80,20)(70,10) \qbezier(90,30)(80,40)(70,50)
\qbezier(30,10)(50,10)(70,10) \qbezier(30,50)(50,50)(70,50)
\end{picture}
~~~
\begin{picture}(100,60)
\put(10,30){\circle*{3}}
\put(30,10){\circle*{3}}    \put(30,50){\circle*{3}}
\put(50,30){\circle*{3}}
\put(70,10){\circle*{3}}    \put(70,50){\circle*{3}}
\put(90,30){\circle*{3}}
\qbezier(50,30)(40,20)(30,10) \qbezier(50,30)(40,40)(30,50)
\qbezier(50,30)(60,20)(70,10) \qbezier(50,30)(60,40)(70,50)
\qbezier(10,30)(20,20)(30,10) \qbezier(10,30)(20,40)(30,50)
\qbezier(90,30)(80,20)(70,10) \qbezier(90,30)(80,40)(70,50)
\qbezier(30,10)(30,30)(30,50) \qbezier(70,10)(70,30)(70,50)
\end{picture}
~~~
\begin{picture}(100,60)
\put(5,15){\circle*{3}}     \put(5,45){\circle*{3}}
\put(35,15){\circle*{3}}    \put(35,45){\circle*{3}}
\put(65,15){\circle*{3}}    \put(65,45){\circle*{3}}
\put(95,15){\circle*{3}}    \put(95,45){\circle*{3}}
\qbezier(5,15)(20,15)(35,15)   \qbezier(5,45)(20,45)(35,45)
\qbezier(35,15)(50,15)(65,15)   \qbezier(35,45)(50,45)(65,45)
\qbezier(65,15)(80,15)(95,15)   \qbezier(65,45)(80,45)(95,45)
\qbezier(35,15)(50,30)(65,45)   \qbezier(35,45)(50,30)(65,15)
\qbezier(65,15)(65,30)(65,45)   \qbezier(95,15)(95,30)(95,45)
\end{picture}
~~~
\begin{picture}(100,60)
\put(10,30){\circle*{3}}
\put(30,10){\circle*{3}}    \put(30,50){\circle*{3}}
\put(50,30){\circle*{3}}
\put(70,10){\circle*{3}}    \put(70,50){\circle*{3}}
\put(90,30){\circle*{3}}
\put(50,10){\circle*{3}}    \put(50,50){\circle*{3}}
\qbezier(50,30)(40,20)(30,10) \qbezier(50,30)(40,40)(30,50)
\qbezier(50,30)(60,20)(70,10) \qbezier(50,30)(60,40)(70,50)
\qbezier(10,30)(20,20)(30,10) \qbezier(10,30)(20,40)(30,50)
\qbezier(90,30)(80,20)(70,10) \qbezier(90,30)(80,40)(70,50)
\qbezier(30,10)(50,10)(70,10) \qbezier(30,50)(50,50)(70,50)
\qbezier(50,10)(40,30)(50,50)
\end{picture}

\medskip

%\begin{multicols}{2}
\begin{picture}(100,60)
\put(10,30){\circle*{3}}
\put(30,10){\circle*{3}}    \put(30,50){\circle*{3}}
\put(50,30){\circle*{3}}
\put(70,10){\circle*{3}}    \put(70,50){\circle*{3}}
\put(90,30){\circle*{3}}
\put(30,30){\circle*{3}}    \put(70,30){\circle*{3}}
\qbezier(50,30)(40,20)(30,10) \qbezier(50,30)(40,40)(30,50)
\qbezier(50,30)(60,20)(70,10) \qbezier(50,30)(60,40)(70,50)
\qbezier(10,30)(20,20)(30,10) \qbezier(10,30)(20,40)(30,50)
\qbezier(90,30)(80,20)(70,10) \qbezier(90,30)(80,40)(70,50)
\qbezier(30,10)(30,30)(30,50) \qbezier(70,10)(70,30)(70,50)
\qbezier(30,30)(50,20)(70,30)
\end{picture}
~~~
\begin{picture}(100,60)
\put(10,10){\circle*{3}}    \put(10,50){\circle*{3}}
\put(50,10){\circle*{3}}    \put(50,50){\circle*{3}}
\put(90,10){\circle*{3}}    \put(90,50){\circle*{3}}
\put(25,30){\circle*{3}}    \put(65,30){\circle*{3}}    \put(80,30){\circle*{3}}
\qbezier(10,10)(10,30)(10,50)   \qbezier(50,10)(50,30)(50,50)   \qbezier(90,10)(90,30)(90,50)
\qbezier(10,10)(30,10)(50,10)   \qbezier(90,10)(70,10)(50,10)
\qbezier(10,50)(30,50)(50,50)   \qbezier(90,50)(70,50)(50,50)
\qbezier(10,10)(17.5,20)(25,30) \qbezier(10,50)(17.5,40)(25,30)
\qbezier(50,10)(57.5,20)(65,30) \qbezier(50,50)(57.5,40)(65,30)    \qbezier(80,30)(72.5,30)(65,30)
\end{picture}
~~~
\begin{picture}(100,60)
\put(20,10){\circle*{3}}    \put(80,10){\circle*{3}}
\put(5,25){\circle*{3}}    \put(35,25){\circle*{3}}
\put(65,25){\circle*{3}}    \put(95,25){\circle*{3}}
\put(50,40){\circle*{3}}
\put(35,55){\circle*{3}}    \put(65,55){\circle*{3}}
\qbezier(20,10)(12.5,17.5)(5,25)    \qbezier(20,10)(27.5,17.5)(35,25)
\qbezier(80,10)(72.5,17.5)(65,25)    \qbezier(80,10)(87.5,17.5)(95,25)
\qbezier(5,25)(20,25)(35,25)
\qbezier(35,25)(50,25)(65,25)
\qbezier(65,25)(80,25)(95,25)
\qbezier(35,25)(42.5,32.5)(50,40)    \qbezier(65,25)(57.5,32.5)(50,40)
\qbezier(35,55)(42.5,47.5)(50,40)    \qbezier(65,55)(57.5,47.5)(50,40)
\qbezier(35,55)(50,55)(65,55)
\end{picture}
~~~
\begin{picture}(100,60)
\put(20,10){\circle*{3}}    \put(80,10){\circle*{3}}
\put(5,25){\circle*{3}}    \put(35,25){\circle*{3}}
\put(65,25){\circle*{3}}    \put(95,25){\circle*{3}}
\put(50,40){\circle*{3}}
\put(35,55){\circle*{3}}    \put(65,55){\circle*{3}}
\qbezier(20,10)(57.5,17.5)(95,25)    \qbezier(20,10)(27.5,17.5)(35,25)
\qbezier(80,10)(72.5,17.5)(65,25)    \qbezier(80,10)(42.5,17.5)(5,25)
\qbezier(5,25)(20,25)(35,25)
\qbezier(35,25)(50,25)(65,25)
\qbezier(65,25)(80,25)(95,25)
\qbezier(35,25)(42.5,32.5)(50,40)    \qbezier(65,25)(57.5,32.5)(50,40)
\qbezier(35,55)(42.5,47.5)(50,40)    \qbezier(65,55)(57.5,47.5)(50,40)
\qbezier(35,55)(50,55)(65,55)
\end{picture}

\medskip

%\begin{multicols}{2}
\begin{picture}(100,60)
\put(20,10){\circle*{3}}    \put(80,10){\circle*{3}}
\put(5,25){\circle*{3}}    \put(35,25){\circle*{3}}
\put(65,25){\circle*{3}}    \put(95,25){\circle*{3}}
\put(50,40){\circle*{3}}
\put(35,55){\circle*{3}}    \put(65,55){\circle*{3}}
\qbezier(35,55)(20,40)(5,25)    \qbezier(20,10)(27.5,17.5)(35,25)
\qbezier(80,10)(72.5,17.5)(65,25)    \qbezier(80,10)(50,10)(20,10)
\qbezier(5,25)(20,25)(35,25)
\qbezier(35,25)(50,25)(65,25)
\qbezier(65,25)(80,25)(95,25)
\qbezier(35,25)(42.5,32.5)(50,40)    \qbezier(65,25)(57.5,32.5)(50,40)
\qbezier(35,55)(42.5,47.5)(50,40)    \qbezier(65,55)(57.5,47.5)(50,40)
\qbezier(95,25)(80,40)(65,55)
\end{picture}
~~~
\begin{picture}(100,60)
\put(5,5){\circle*{3}}    \put(35,5){\circle*{3}}   \put(65,5){\circle*{3}}    \put(95,5){\circle*{3}}
\put(5,30){\circle*{3}}    \put(35,30){\circle*{3}}   \put(65,30){\circle*{3}}    \put(95,30){\circle*{3}}
\put(35,55){\circle*{3}}    \put(65,55){\circle*{3}}
\qbezier(5,5)(5,17.5)(5,30)     \qbezier(35,5)(35,17.5)(35,30)
\qbezier(65,5)(65,17.5)(65,30)     \qbezier(95,5)(95,17.5)(95,30)
\qbezier(5,30)(20,42.5)(35,55)     \qbezier(5,30)(35,42.5)(65,55)
\qbezier(35,30)(35,42.5)(35,55)     \qbezier(35,30)(50,42.5)(65,55)
\qbezier(65,30)(50,42.5)(35,55)     \qbezier(65,30)(65,42.5)(65,55)
\qbezier(95,30)(65,42.5)(35,55)     \qbezier(95,30)(80,42.5)(65,55)
\end{picture}
~~~
\begin{picture}(100,60)
\put(5,5){\circle*{3}}    \put(35,5){\circle*{3}}   \put(65,5){\circle*{3}}    \put(95,5){\circle*{3}}
\put(5,30){\circle*{3}}    \put(35,30){\circle*{3}}   \put(65,30){\circle*{3}}    \put(95,30){\circle*{3}}
\put(35,55){\circle*{3}}    \put(65,55){\circle*{3}}
\qbezier(5,5)(5,17.5)(5,30)     \qbezier(35,5)(35,17.5)(35,30)
\qbezier(65,5)(65,17.5)(65,30)     \qbezier(95,5)(95,17.5)(95,30)
\qbezier(5,30)(20,30)(35,30)     \qbezier(35,30)(50,30)(65,30)      \qbezier(65,30)(80,30)(95,30)
\qbezier(35,30)(35,42.5)(35,55)     \qbezier(65,30)(65,42.5)(65,55)     \qbezier(35,55)(50,55)(65,55)
\qbezier(5,5)(20,5)(35,5)   \qbezier(5,30)(20,17.5)(35,5)
\qbezier(65,5)(80,5)(95,5)   \qbezier(95,30)(80,17.5)(65,5)
\end{picture}
~~~
\begin{picture}(100,60)
\put(5,5){\circle*{3}}    \put(35,5){\circle*{3}}   \put(65,5){\circle*{3}}    \put(95,5){\circle*{3}}
\put(5,30){\circle*{3}}    \put(35,30){\circle*{3}}   \put(65,30){\circle*{3}}    \put(95,30){\circle*{3}}
\put(35,55){\circle*{3}}    \put(65,55){\circle*{3}}
\qbezier(5,5)(5,17.5)(5,30)     \qbezier(35,5)(35,17.5)(35,30)
\qbezier(65,5)(65,17.5)(65,30)     \qbezier(95,5)(95,17.5)(95,30)
\qbezier(5,30)(20,30)(35,30)     \qbezier(35,30)(50,30)(65,30)      \qbezier(65,30)(80,30)(95,30)
\qbezier(35,30)(35,42.5)(35,55)     \qbezier(65,30)(65,42.5)(65,55)     \qbezier(35,55)(50,55)(65,55)
\qbezier(5,5)(20,5)(35,5)   \qbezier(95,30)(65,17.5)(35,5)
\qbezier(65,5)(80,5)(95,5)   \qbezier(5,30)(35,17.5)(65,5)
\end{picture}

\medskip

%\begin{multicols}{2}
\begin{picture}(100,60)
\put(5,5){\circle*{3}}    \put(35,5){\circle*{3}}   \put(65,5){\circle*{3}}    \put(95,5){\circle*{3}}
\put(5,30){\circle*{3}}    \put(35,30){\circle*{3}}   \put(65,30){\circle*{3}}    \put(95,30){\circle*{3}}
\put(35,55){\circle*{3}}    \put(65,55){\circle*{3}}
\qbezier(5,5)(5,17.5)(5,30)     \qbezier(35,5)(35,17.5)(35,30)
\qbezier(65,5)(65,17.5)(65,30)     \qbezier(95,5)(95,17.5)(95,30)
\qbezier(5,30)(20,30)(35,30)     \qbezier(35,30)(50,30)(65,30)      \qbezier(65,30)(80,30)(95,30)
\qbezier(35,30)(35,42.5)(35,55)     \qbezier(65,30)(65,42.5)(65,55)     \qbezier(35,55)(50,55)(65,55)
\qbezier(95,5)(65,-5)(35,5)   \qbezier(5,30)(20,17.5)(35,5)
\qbezier(65,5)(35,-5)(5,5)   \qbezier(95,30)(80,17.5)(65,5)
\end{picture}
~~~
\begin{picture}(100,60)
\put(5,5){\circle*{3}}    \put(35,5){\circle*{3}}   \put(65,5){\circle*{3}}    \put(95,5){\circle*{3}}
\put(5,30){\circle*{3}}    \put(35,30){\circle*{3}}   \put(65,30){\circle*{3}}    \put(95,30){\circle*{3}}
\put(35,55){\circle*{3}}    \put(65,55){\circle*{3}}
\qbezier(5,5)(5,17.5)(5,30)     \qbezier(35,5)(35,17.5)(35,30)
\qbezier(65,5)(65,17.5)(65,30)     \qbezier(95,5)(95,17.5)(95,30)
\qbezier(5,30)(20,30)(35,30)     \qbezier(35,30)(50,30)(65,30)      \qbezier(65,30)(80,30)(95,30)
\qbezier(35,30)(35,42.5)(35,55)     \qbezier(65,30)(65,42.5)(65,55)     \qbezier(35,55)(50,55)(65,55)
\qbezier(95,5)(65,-5)(35,5)   \qbezier(95,30)(65,17.5)(35,5)
\qbezier(65,5)(35,-5)(5,5)   \qbezier(5,30)(35,17.5)(65,5)
\end{picture}
%\end{multicols}

\bigskip
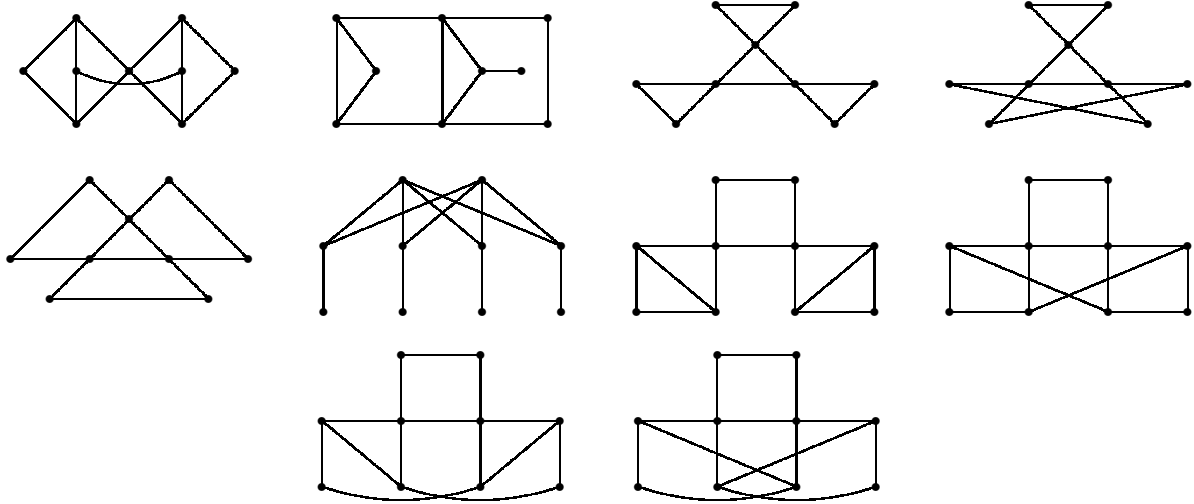
{\bf Figure 4:} Pseudo $3$-regular graphs (up to isomorphism) of order at most $10$.
\end{center}

\bigskip

\section{Minimal pseudo regular graphs}     \label{sec_min_prg}
Recall that $N(k)$ is the minimum number of vertices in a pseudo $k$-regular graph. The pseudo $k$-regular graphs of order $N(k)$ are studied in this section for $2\leq k\leq 7.$ Then we prove that a pseudo regular $k$-regular graph of every order $n\geq N(k)$ always exists for $k=3,4.$ Lower bounds and upper bounds for $N(k)$ are also proposed.

\medskip

\subsection{Pseudo $k$-regular graphs of order $N(k)$ for $2\leq k\leq 7$}
By Lemma~\ref{t2} the tree $T_2$ given in Definition~\ref{defn_Tk} is the only pseudo $2$-regular graph, and thus $N(2)=7.$
By checking the list of graphs in Figure 4, we find $N(3)=7$ and the two graphs in Figure 1 are the only pseudo $3$-regular graphs of order $N(3).$
As an application of Proposition~\ref{prop_AnalogueEG}, $N(k)$ and all of the possible degree sequences of a pseudo $k$-regular graph on $N(k)$ vertices are obtained for $4\leq k\leq 7$ by programming, which are listed in the following table. Note that when $k=5$ the possible degree sequences are not unique.
\begin{center}
\begin{tabular}{l|c|l}
    $k$ & $N(k)$ & Possible degree sequences    \\ \hline
    $2$ & $7$ & $3,2,2,2,1,1,1$      \\
    $3$ & $7$ & $4,3,3,3,3,2,2$    \\
    $4$ & $8$ & $5,5,4,4,4,3,3,2$    \\
    $5$ & $9$ & $6,6,6,5,5,4,4,4,2$   \\
        &     & $6,6,5,5,5,5,4,4,4$     \\
    $6$ & $11$ & $8,6,6,6,6,6,6,6,6,4,4$    \\
    $7$ & $11$ & $8,8,8,7,7,7,7,6,6,6,6$
\end{tabular}

\bigskip

{\bf Table 1:} The degree sequences of pseudo $k$-regular graphs of order $N(k)$ for $2\leq k\leq 7.$
\end{center}

\medskip

For $4\leq k\leq 7,$ we also find the pseudo $k$-regular graphs of order $N(k)$. Note that when $k=4$ the corresponding graph of order $N(4)=8$ is unique up to isomorphism. While $k=5,$ there are three non-isomorphic pseudo $5$-regular graphs of order $N(5)=9.$ For $k=6,$ the pseudo $6$-regular on $N(6)=11$ vertices are given by a $4$-regular graph on $8$ vertices with one vertex adjacent to all of them, and each of the remained two vertices adjacent to $4$ of them with no common neighbor. The graphs mentioned above are shown in Figure 5. As $k=7,$ four non-isomorphic pseudo $7$-regular graphs of order $N(7)=11$ are attained by exhaustedly graphing the degree sequence given in Table 1. In fact, we will give a construction of pseudo $k$-regular graphs of order $k+4$ in Proposition~\ref{prop_upbd} for every odd $k.$

\smallskip

\begin{center}
\begin{tikzpicture}
\draw[fill] (-1.5,1.5)coordinate(A1) circle(0.1cm);
\draw[fill] (1.5,1.5)coordinate(A2) circle(0.1cm);
\draw[fill] (0,0.75)coordinate(A3) circle(0.1cm);
\draw[fill] (-1.5,0)coordinate(A4) circle(0.1cm);
\draw[fill] (1.5,0)coordinate(A5) circle(0.1cm);
\draw[fill] (-1.5,-1.5)coordinate(A6) circle(0.1cm);
\draw[fill] (1.5,-1.5)coordinate(A7) circle(0.1cm);
\draw[fill] (0,-2)coordinate(A8) circle(0.1cm);
\draw(A1)--(A5)--(A2)--(A4)--(A1)--(A2);
\draw(A4)--(A5)--(A7)--(A4)--(A6)--(A5);
\draw(A6)--(A7)--(A8)--(A6);
\end{tikzpicture}
~~~
\begin{tikzpicture}
\draw[fill] (-1.5,1.5)coordinate(A1) circle(0.1cm);
\draw[fill] (0,1.75)coordinate(A2) circle(0.1cm);
\draw[fill] (1.5,1.5)coordinate(A3) circle(0.1cm);
\draw[fill] (-1.5,0)coordinate(A4) circle(0.1cm);
\draw[fill] (0,-0.25)coordinate(A5) circle(0.1cm);
\draw[fill] (1.5,0)coordinate(A6) circle(0.1cm);
\draw[fill] (-1.5,-1.5)coordinate(A7) circle(0.1cm);
\draw[fill] (1.5,-1.5)coordinate(A8) circle(0.1cm);
\draw[fill] (0,-2)coordinate(A9) circle(0.1cm);
\draw(A1)--(A2)--(A3)--(A1)--(A5)--(A3)--(A6)--(A5)--(A4)--(A6)--(A2)--(A4)--(A1);
\draw(A4)--(A8)--(A6)--(A7)--(A5)--(A8)--(A7)--(A4);
\draw(A7)--(A9)--(A8);
\end{tikzpicture}
~~~
\begin{tikzpicture}
\draw[] plot[samples=500,domain=-90:270,variable=\t]({cos(\t)},{0.5*sin(\t)});
\draw[] (0,-0.7)coordinate(N) node{2 matchings};
\draw[fill] (-1.2,2)coordinate(A1) circle(0.1cm);
\draw[fill] (1.2,2)coordinate(A2) circle(0.1cm);
%{\draw[fill,shift={(-4,-4)}] (\x,0) circle(0.1cm);}
\draw[fill] (-0.75,0)coordinate(B1) circle(0.1cm);
\draw[fill] (-0.25,0)coordinate(B2) circle(0.1cm);
\draw[fill] (0.25,0)coordinate(B3) circle(0.1cm);
\draw[fill] (0.75,0)coordinate(B4) circle(0.1cm);
\draw[fill] (-1.2,-2)coordinate(C1) circle(0.1cm);
\draw[fill] (0,-1.5)coordinate(C2) circle(0.1cm);
\draw[fill] (1.2,-2)coordinate(C3) circle(0.1cm);
\draw(A1)--(B1)--(A2)--(B2)--(A1)--(B3)--(A2)--(B4)--(A1)--(A2)--(C3)--(C1)--(A1);
\draw(B1)--(C1)--(B2)--(C2)--(B1)--(B2);
\draw(B3)--(C3)--(B4)--(C2)--(B3)--(B4);
\end{tikzpicture}
~~~
\begin{tikzpicture}
\draw[] plot[samples=500,domain=-90:270,variable=\t]({2*cos(\t)},{0.5*sin(\t)});
\draw[] (0,-0.7)coordinate(N) node{4-regular};
\draw[fill] (0,2)coordinate(A) circle(0.1cm);
%{\draw[fill,shift={(-4,-4)}] (\x,0) circle(0.1cm);}
\draw[fill] (-1.75,0)coordinate(B1) circle(0.1cm);
\draw[fill] (-1.25,0)coordinate(B2) circle(0.1cm);
\draw[fill] (-0.75,0)coordinate(B3) circle(0.1cm);
\draw[fill] (-0.25,0)coordinate(B4) circle(0.1cm);
\draw[fill] (0.25,0)coordinate(B5) circle(0.1cm);
\draw[fill] (0.75,0)coordinate(B6) circle(0.1cm);
\draw[fill] (1.25,0)coordinate(B7) circle(0.1cm);
\draw[fill] (1.75,0)coordinate(B8) circle(0.1cm);
\draw[fill] (-1,-2)coordinate(C1) circle(0.1cm);
\draw[fill] (1,-2)coordinate(C2) circle(0.1cm);
\draw(A)--(B1)--(C1)--(B2)--(A)--(B3)--(C1)--(B4)--(A)--(B5)--(C2)--(B6)--(A)--(B7)--(C2)--(B8)--(A);
\end{tikzpicture}

\bigskip

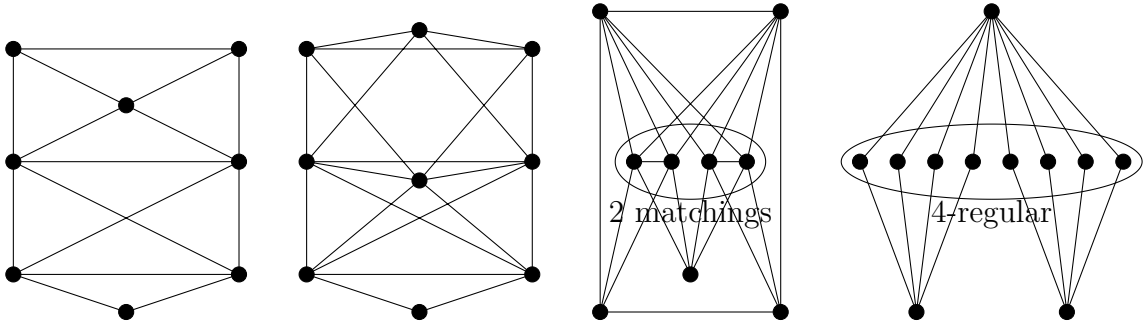
{\bf Figure 5:} The pseudo $k$-regular graphs of order $N(k)$ for $k=4,5,6.$
\end{center}

\medskip

\subsection{Existence of pseudo $k$-regular graphs of arbitrary orders}
It is of some interest that whether there exists a pseudo $k$-regular graph of every order $n\geq N(k)$. We introduce a method to construct a pseudo $k$-regular graph of larger order from a known one. The basic idea can be referred to Figure 4. One can observe that the graph at the upper right corner of Figure 4 can be obtained from that at the upper left corner by finding two edges with two $3$-degree ends, inserting one vertex in each of them, and joining these two new vertices. Consequently, we conclude the following result.
\begin{lemma}       \label{lem_small_pkreg}
If there exists a pseudo $k$-regular graph on $n$ vertices which has $k-1$ edges with two $k$-degree end vertices, then there exists a pseudo $k$-regular graph on $n+m(k-1)$ vertices for any positive integer $m.$
\end{lemma}
\begin{proof}
Suppose that $e_1,e_2,\ldots,e_{k-1}$ are the edges with two $k$-degree ends. For $i=1,2,\ldots,k-1,$ insert one vertex $v_i$ in $e_i$ by deleting $e_i$ and joining two edges between $v_i$ and two ends of $e_i,$ respectively. Then, add edges $v_iv_j$ for all $1\leq v_i < v_j\leq k-1.$ It is not difficult to check that the new graph on $n+k-1$ vertices is still pseudo $k$-regular and has no less than $k-1$ edges with two $k$-degree ends. Using the above construction method to generate new larger graphs pseudo $k$-regular repeatedly, we can increase the number of vertices by $k-1$ in each operation.
\end{proof}

\medskip

By some small pseudo $k$-regular graphs and applying Lemma~\ref{lem_small_pkreg}, we can attain a pseudo $k$-regular graph of every order at least $N(k)$ for $k=3,4.$
\begin{corollary}
For $k=3,4$ there exists a pseudo $k$-regular graph on $n$ vertices for every $n\geq N(k).$
\end{corollary}
\begin{proof}
For $k=3,$ we have $N(3)=7.$ By Lemma~\ref{lem_small_pkreg}, the two pseudo $3$-regular graphs given below of orders $7$ and $10$ which have two edges $e_1,e_2$ and $e_3,e_4$ with $3$-degree ends, respectively, provide the existence of pseudo $3$-regular graphs of order $n\geq N(k)$ except for $n=8.$ In particular, a pseudo $3$-regular graph of order $8$ is also shown as follows.
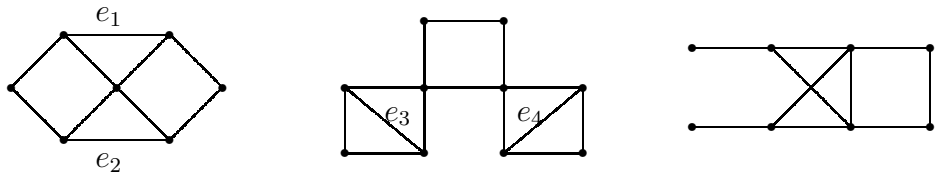
\begin{center}
\begin{picture}(100,60)
\put(10,30){\circle*{3}}
\put(30,10){\circle*{3}}    \put(30,50){\circle*{3}}
\put(50,30){\circle*{3}}
\put(70,10){\circle*{3}}    \put(70,50){\circle*{3}}
\put(90,30){\circle*{3}}
\put(42,55){$e_1$}
\put(42,0){$e_2$}
\qbezier(50,30)(40,20)(30,10) \qbezier(50,30)(40,40)(30,50)
\qbezier(50,30)(60,20)(70,10) \qbezier(50,30)(60,40)(70,50)
\qbezier(10,30)(20,20)(30,10) \qbezier(10,30)(20,40)(30,50)
\qbezier(90,30)(80,20)(70,10) \qbezier(90,30)(80,40)(70,50)
\qbezier(30,10)(50,10)(70,10) \qbezier(30,50)(50,50)(70,50)
\end{picture}
~~~~~~
\begin{picture}(100,60)
\put(5,5){\circle*{3}}    \put(35,5){\circle*{3}}   \put(65,5){\circle*{3}}    \put(95,5){\circle*{3}}
\put(5,30){\circle*{3}}    \put(35,30){\circle*{3}}   \put(65,30){\circle*{3}}    \put(95,30){\circle*{3}}
\put(35,55){\circle*{3}}    \put(65,55){\circle*{3}}
\put(20,17.5){$e_3$}
\put(70,17.5){$e_4$}
\qbezier(5,5)(5,17.5)(5,30)     \qbezier(35,5)(35,17.5)(35,30)
\qbezier(65,5)(65,17.5)(65,30)     \qbezier(95,5)(95,17.5)(95,30)
\qbezier(5,30)(20,30)(35,30)     \qbezier(35,30)(50,30)(65,30)      \qbezier(65,30)(80,30)(95,30)
\qbezier(35,30)(35,42.5)(35,55)     \qbezier(65,30)(65,42.5)(65,55)     \qbezier(35,55)(50,55)(65,55)
\qbezier(5,5)(20,5)(35,5)   \qbezier(5,30)(20,17.5)(35,5)
\qbezier(65,5)(80,5)(95,5)   \qbezier(95,30)(80,17.5)(65,5)
\end{picture}
~~~~~~
\begin{picture}(100,60)
\put(5,15){\circle*{3}}     \put(5,45){\circle*{3}}
\put(35,15){\circle*{3}}    \put(35,45){\circle*{3}}
\put(65,15){\circle*{3}}    \put(65,45){\circle*{3}}
\put(95,15){\circle*{3}}    \put(95,45){\circle*{3}}
\qbezier(5,15)(20,15)(35,15)   \qbezier(5,45)(20,45)(35,45)
\qbezier(35,15)(50,15)(65,15)   \qbezier(35,45)(50,45)(65,45)
\qbezier(65,15)(80,15)(95,15)   \qbezier(65,45)(80,45)(95,45)
\qbezier(35,15)(50,30)(65,45)   \qbezier(35,45)(50,30)(65,15)
\qbezier(65,15)(65,30)(65,45)   \qbezier(95,15)(95,30)(95,45)
\end{picture}
\end{center}

\smallskip

For $k=4,$ we have $N(4)=8.$ By Lemma~\ref{lem_small_pkreg}, the three pseudo $4$-regular graphs given below of orders $9,$ $10,$ and $11$ which have three edges $e_1,e_2,e_3,$ $e_4,e_5,e_6,$ and $e_7,e_8,e_9$ with $4$-degree ends, respectively, provide the existence of pseudo $4$-regular graphs of order $n\geq N(k)$ except for $n=8.$ In particular, a pseudo $4$-regular graph of order $8$ is also shown as follows.
%For $k=4,$ we have $N(4)=8.$ by Lemma~\ref{lem_small_pkreg} it is sufficient to give one pseudo $4$-regular graph of order $8,$ and three of order $9,$ $10,$ and $11$ which have three edges $e_1,e_2,e_3,$ $e_4,e_5,e_6,$ and $e_7,e_8,e_9$ with $4$-degree ends, respectively. The four graphs are provided in the following.
%
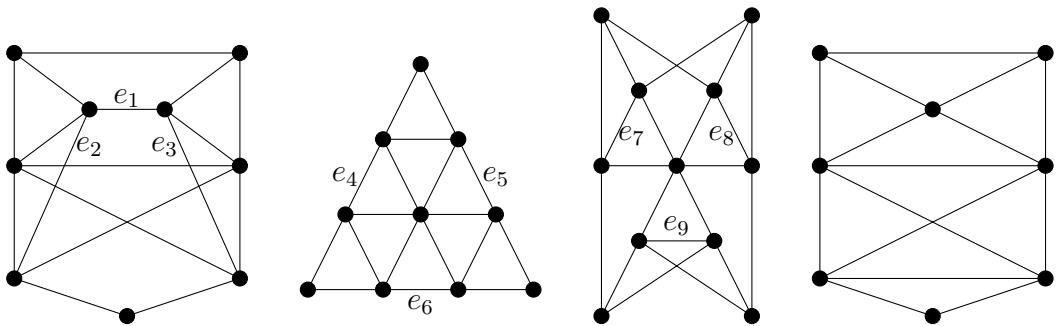
\begin{center}
\begin{tikzpicture}
\draw[fill] (-1.5,1.5)coordinate(A1) circle(0.1cm);
\draw[fill] (1.5,1.5)coordinate(A2) circle(0.1cm);
\draw[fill] (-0.5,0.75)coordinate(A3) circle(0.1cm);
\draw[fill] (0.5,0.75)coordinate(A4) circle(0.1cm);
\draw[fill] (-1.5,0)coordinate(A5) circle(0.1cm);
\draw[fill] (1.5,0)coordinate(A6) circle(0.1cm);
\draw[fill] (-1.5,-1.5)coordinate(A7) circle(0.1cm);
\draw[fill] (1.5,-1.5)coordinate(A8) circle(0.1cm);
\draw[fill] (0,-2)coordinate(A9) circle(0.1cm);
\draw[] (0,0.9)coordinate(e1) node{$e_1$};
\draw[] (-0.5,0.25)coordinate(e2) node{$e_2$};
\draw[] (0.5,0.25)coordinate(e3) node{$e_3$};
\draw(A1)--(A2)--(A4)--(A3)--(A1)--(A5)--(A6)--(A2);
\draw(A3)--(A5)--(A7)--(A9)--(A8)--(A6)--(A4)--(A8)--(A5);
\draw(A3)--(A7)--(A6);
\end{tikzpicture}
~~~
\begin{tikzpicture}
\draw[fill] (0,1.5)coordinate(A1) circle(0.1cm);
\draw[fill] (-0.5,0.5)coordinate(A2) circle(0.1cm);
\draw[fill] (0.5,0.5)coordinate(A3) circle(0.1cm);
\draw[fill] (-1,-0.5)coordinate(A4) circle(0.1cm);
\draw[fill] (0,-0.5)coordinate(A5) circle(0.1cm);
\draw[fill] (1,-0.5)coordinate(A6) circle(0.1cm);
\draw[fill] (-1.5,-1.5)coordinate(A7) circle(0.1cm);
\draw[fill] (-0.5,-1.5)coordinate(A8) circle(0.1cm);
\draw[fill] (0.5,-1.5)coordinate(A9) circle(0.1cm);
\draw[fill] (1.5,-1.5)coordinate(A10) circle(0.1cm);
\draw[] (-1,0)coordinate(e4) node{$e_4$};
\draw[] (1,0)coordinate(e5) node{$e_5$};
\draw[] (0,-1.7)coordinate(e6) node{$e_6$};
\draw(A1)--(A7)--(A10)--(A1);
\draw(A2)--(A3)--(A8)--(A4)--(A6)--(A9)--(A2);
\end{tikzpicture}
~~~
\begin{tikzpicture}
\draw[fill] (-1,2)coordinate(A1) circle(0.1cm);
\draw[fill] (1,2)coordinate(A2) circle(0.1cm);
\draw[fill] (-0.5,1)coordinate(A3) circle(0.1cm);
\draw[fill] (0.5,1)coordinate(A4) circle(0.1cm);
\draw[fill] (-1,0)coordinate(A5) circle(0.1cm);
\draw[fill] (0,0)coordinate(A6) circle(0.1cm);
\draw[fill] (1,0)coordinate(A7) circle(0.1cm);
\draw[fill] (-0.5,-1)coordinate(A8) circle(0.1cm);
\draw[fill] (0.5,-1)coordinate(A9) circle(0.1cm);
\draw[fill] (-1,-2)coordinate(A10) circle(0.1cm);
\draw[fill] (1,-2)coordinate(A11) circle(0.1cm);
\draw[] (-0.6,0.4)coordinate(e7) node{$e_7$};
\draw[] (0.6,0.4)coordinate(e8) node{$e_8$};
\draw[] (0,-0.8)coordinate(e9) node{$e_9$};
\draw(A1)--(A11)--(A2)--(A10)--(A1)--(A4)--(A7)--(A5)--(A3)--(A2);
\draw(A10)--(A9)--(A8)--(A11);
\end{tikzpicture}
~~~
\begin{tikzpicture}
\draw[fill] (-1.5,1.5)coordinate(A1) circle(0.1cm);
\draw[fill] (1.5,1.5)coordinate(A2) circle(0.1cm);
\draw[fill] (0,0.75)coordinate(A3) circle(0.1cm);
\draw[fill] (-1.5,0)coordinate(A4) circle(0.1cm);
\draw[fill] (1.5,0)coordinate(A5) circle(0.1cm);
\draw[fill] (-1.5,-1.5)coordinate(A6) circle(0.1cm);
\draw[fill] (1.5,-1.5)coordinate(A7) circle(0.1cm);
\draw[fill] (0,-2)coordinate(A8) circle(0.1cm);
\draw(A1)--(A5)--(A2)--(A4)--(A1)--(A2);
\draw(A4)--(A5)--(A7)--(A4)--(A6)--(A5);
\draw(A6)--(A7)--(A8)--(A6);
\end{tikzpicture}
\end{center}
\end{proof}

\medskip

\subsection{Bounds on $N(k)$}
In this subsection,
we give some lower and upper bounds for the minimum number $N(k)$ of vertices in a pseudo $k$-regular graph.
\begin{proposition}       \label{prop_lowbd}
For positive integer $k\geq 2,$ the minimum number $N(k)$ of vertices in a pseudo $k$-regular graph satisfies
$$N(k)\geq k+3.$$
\end{proposition}
\begin{proof}
It is direct to see that $N(k)\geq k+2$ since the maximum degree in a pseudo $k$-regular graph is more than $k.$ Assume that there exists a pseudo $k$-regular graph $G$ on $k+2$ vertices. Let $V(G)$ be partitioned into $S\cup T$ such that $S$ collects the vertices of degree $k+1$ and has cardinality $|S|=a.$ Then the neighbors of any vertex $i\in S$ have degree sum
$$(a-1)(k+1)+\sum_{j\in T}d_j = k(k+1) = k(a-1+|T|),$$
and hence
\begin{equation}        \label{eq_lowbd_1}
a-1=\sum_{j\in T}(k-d_j).
\end{equation}
On the other hand, from Lemma~\ref{lem_sum_dimi} we have
$$ak(k+1)+\sum_{j\in T}d_jk = a(k+1)^2+\sum_{j\in T}d_j^2,$$
which implies that
\begin{equation}        \label{eq_lowbd_2}
a(k+1)= \sum_{j\in T}d_j(k-d_j)\leq k\sum_{j\in T}(k-d_j).
\end{equation}
A contradiction occurs by substituting~\eqref{eq_lowbd_1} into~\eqref{eq_lowbd_2}.
%\smallskip
%
%Let $G$ be a pseudo $k$-regular graph of order $k+3.$ Suppose $V(G)$ is partitioned in to $V(G)=S_2\cup S_1\cup S_0\cup T$ such that $S_i$ is the set of vertices of degree $k+i$ for $i=0,1,2$ and $T$ is the set of vertices of degree less than $k.$\\
%\smallskip
%%
%\textbf{Case 1.}
%If $S_2$ is non-empty, then the sum of degrees is $(k+1)(k+2),$ since the neighbors of a vertex in $S_2$ has degree sum $k(k+2).$ If $S_1$ is non-empty, then each vertex $v\in S_1$ is non-adjacent to some vertex of degree $k+1$ (which also implies that $|S_1|\geq 2$) since the neighbors of a vertex in $S_1$ has degree sum $k(k+1).$ Consequently, each vertex in $S_0\cup T$ is adjacent to the all vertices of degree larger than $k,$ and thus of average $2$-degree larger than $k,$ a contradiction.\\
%\smallskip
%%
%\textbf{Case 2.}
%Assume that $S_2=\emptyset.$
\end{proof}

\medskip

We emphasize that $k+3$ is tight by the following example.
\begin{example}
For $k=17,$ a pseudo $17$-regular graph of order $20$ will be given and hence $N(17)=20.$ Let $S,T,$ and $U=\{u,v\}$ be the vertex sets of a complete graph on $12$ vertices, a $2$-regular graph on $6$ vertices, and $2$ isolated vertices. By adding an edge $xy$ for each pair of vertices $x\in T$ and $y\in S\cup U,$ and $12$ edges between $U$ and $S$ such that $u$ is adjacent to $6$ vertices of $S$ and $v$ is adjacent to the other $6$ vertices. One can see that the new graph is pseudo $17$-regular with $20$ vertices.
\end{example}

\medskip

Furthermore, a pseudo $k$-regular graph on $k+4$ and $k+6$ vertices always exists for each odd and even $k,$ respectively.
\begin{proposition}     \label{prop_upbd}
For each positive integer $k\geq 3,$ we have
$$N(k)\leq \left\{
\begin{array}{ll}
  k+4 & \text{if $k$ is odd} \\
  k+6 & \text{if $k$ is even}
\end{array}
\right..$$
\end{proposition}
\begin{proof}
We prove it by providing a pseudo $k$-regular graph of the corresponding order. Suppose that $k$ is odd. Let $p=(k+1)/2$ and $G$ be a graph on $k+4$ vertices with vertex set partition $V(G)=S\cup T\cup U,$ where $S=\{s_i\mid 1\leq i\leq 4\},T=\{t_i\mid 1\leq i\leq p-1\},$ and $U=\{u_i\mid 0\leq i\leq p-1\}.$ Let the edge set of the induced subgraph on $S\cup T$ be
$$\{s_1s_2,s_3s_4\}\cup\{xy\mid x\in T~\text{and}~y\in S\cup T\setminus\{x\}\},$$
the edge set of the induced subgraph on $U$ be
$$\{u_iu_j\mid 1 \leq i+1 < j \leq p-1\},$$
the edge set between $T$ and $U$ be
$$\{t_iu_j\mid 1\leq i\leq p-1~\text{and}~j=i+1,i+2,\ldots,i+p-2~\text{(mod $p$)}\},$$
and the edge set between $S$ and $U$ be
$$\{xy\mid x\in S, y\in U\}\setminus\{s_1u_0,s_2u_0,s_3u_{p-1},s_4u_{p-1}\}.$$
We can see that $G$ is pseudo $k$-regular.

\smallskip

For even $k,$ let $p=k/2+1$ and $G$ be a graph on $k+6$ vertices with vertex set partition $V(G)=S\cup T\cup U,$ where $S=\{s_i\mid 1\leq i\leq 6\},T=\{t_i\mid 1\leq i\leq p-2\},$ and $U=\{u_i\mid 0\leq i\leq p-1\}.$ Let the edge set of the induced subgraph on $S\cup T$ be
$$\{s_1s_2,s_2s_3,s_3s_4,s_4s_5,s_5s_6,s_6s_1\}\cup\{xy\mid x\in T~\text{and}~y\in S\cup T\setminus\{x\}\},$$
the edge set of the induced subgraph on $U$ be
$$\{u_iu_j\mid 2 \leq i+2 < j \leq p-1\},$$
the edge set between $T$ and $U$ be
$$\{t_iu_j\mid 1\leq i\leq p-2~\text{and}~j=i+2,i+3,\ldots,i+p-2~\text{(mod $p$)}\},$$
and the edge set between $S$ and $U$ be
\begin{multline*}
\{xy\mid x\in S, y\in U\}
\\
\setminus\{s_1u_0,s_1u_1,s_2u_0,s_2u_1,s_3u_0,s_3u_{p-1},s_4u_0,s_4u_{p-1},s_5u_{p-2},s_5u_{p-1},s_6u_{p-2},s_6u_{p-1}\}.
\end{multline*}
One can check that the pseudo $k$-regular property holds for $G.$
\end{proof}

\bigskip

\section{Open problems}     \label{sec_open}
In this paper, we propose conditions and applications for the sequence of degree pairs, and further utilize them to find pseudo $k$-regular graphs.
%The pseudo $3$-regular graphs of at most $10$ vertices are given, and the minimum number $N(k)$ of vertices in a pseudo $k$-regular graph is determined as $N(2)=7,$ $N(3)=7,$ $N(4)=8,$ $N(5)=9,$ $N(6)=11,$ and $N(7)=11$.
To conclude this paper, several open problems are listed in the following.
\begin{enumerate}
\item       Give a necessary and sufficient condition for a sequence of positive integers that can be the degree sequence of a finite pseudo $k$-regular graph with no isolated vertices for every positive integer $k$.
\item       Give a necessary and sufficient condition for a sequence of pairs of positive real numbers that is graphic on a finite simple graph with no isolated vertices.
\item       Is $N(k)$ non-decreasing? It is true for $k\leq 7.$
\item       For each positive integer $k\geq 8,$ determine $N(k)$, and find all pseudo $k$-regular graphs of order $N(k)$.
\item       Does there always exist a pseudo $k$-regular graph on $n$ vertices for any positive integers $k\geq 5$ and $n\geq N(k)$?
\item       Give a function $g(n,k)$ for positive integers $n,k$ that maps to the number of pseudo $k$-regular graphs of order $n$ up to isomorphism. Currently we have that $g(n,3)=0$ for $n\leq 6$ and $g(7,3)=2;$ $g(n,4)=0$ for $n\leq 7$ and $g(8,4)=1;$ $g(n,5)=0$ for $n\leq 8$ and $g(9,5)=3;$ $g(n,6)=0$ for $n\leq 10$; $g(n,7)=0$ for $n\leq 10$ and $g(11,7)=4.$
\end{enumerate}
\section*{Acknowledgments}
The authors thank Hou, Yaoping for pointing out that the $k$-harmonic graphs were previously studied in \cite{03:dgr}.
This research is supported by the Teacher Research Capacity Promotion Program of Beijing Normal University Zhuhai, the National Natural Science Foundation of China (NSFC) with Grant No. 11871107, and the Ministry of Science and Technology of Taiwan R.O.C. under the projects MOST 107-2917-I-564-004 and MOST 107-2115-M-009-009-MY2.

\end{document}